\documentclass[11pt,reqno]{amsart}
\usepackage{amssymb}
\usepackage{amsmath}
\usepackage{amscd}
\begin{document}
\setlength{\oddsidemargin}{0cm} \setlength{\evensidemargin}{0cm}

\theoremstyle{plain}
\newtheorem{theorem}{Theorem}[section]
\newtheorem{proposition}[theorem]{Proposition}
\newtheorem{lemma}[theorem]{Lemma}
\newtheorem{corollary}[theorem]{Corollary}
\newtheorem{conj}[theorem]{Conjecture}

\theoremstyle{definition}
\newtheorem{definition}[theorem]{Definition}
\newtheorem{exam}[theorem]{Example}
\newtheorem{remark}[theorem]{Remark}

\numberwithin{equation}{section}

\title[Pseudo-Riemannian weakly symmetric manifolds]
{Pseudo-Riemannian weakly symmetric manifolds}

\author{Zhiqi Chen}
\address{School of Mathematical Sciences and LPMC \\ Nankai University \\ Tianjin 300071, P.R. China} \email{chenzhiqi@nankai.edu.cn}

\author{Joseph A. Wolf}
\address{Department of Mathematics \\ University of California, Berkeley \\ CA 94720--3840, U.S.A.} \email{jawolf@math.berkeley.edu}

\subjclass[2000]{}

\keywords{}

\begin{abstract}
There is a well developed theory of weakly symmetric Riemannian manifolds.  
Here it is shown that several results in the Riemannian case are also
valid for weakly symmetric pseudo-Riemannian manifolds, but some
require additional hypotheses.  The topics discussed are homogeneity,
geodesic completeness, the geodesic orbit property, weak symmetries, and
the structure of the nilradical of the isometry group.  Also, we give a number 
of examples of weakly symmetric pseudo-Riemannian manifolds, some mirroring 
the Riemannian case and some indicating the problems in extending Riemannian 
results to weakly symmetric pseudo-Riemannian spaces. 
\end{abstract}

\maketitle

\setcounter{section}{-1}
\section{Introduction}

There have been several important extensions of the theory of
Riemannian symmetric spaces.  Weakly symmetric spaces, introduced
by A. Selberg \cite{Se1}, play key roles in number theory, Riemannian geometry
and harmonic analysis.  Pseudo-Riemannian symmetric spaces, including
semisimple symmetric spaces, play central but complementary roles in 
number theory, differential geometry and relativity, Lie group
representation theory and harmonic analysis.  Here we study the
common extension of these two branches of symmetric space theory,
that of weakly symmetric pseudo-Riemannian manifolds.  
\medskip

It is surprising that the theory of weakly symmetric pseudo-Riemannian 
manifolds has so many open problems.  Here we recall what is known in
its differential-geometric aspect and prove a few new results.  Some
facts, such as homogeneity for weakly symmetric pseudo-Riemannian spaces,
are easy.  Others, in particular questions of sectional curvature and
the structure of the nilradical of the isometry group, are subtle.
\medskip

In Section \ref{prm} we review a number of basic facts on the geometry
of pseudo-Riemannian 
manifolds, especially concerning completeness and homogeneity, 
pointing out the contrast with the Riemannian case.
\medskip

Section \ref{geom} specializes to the setting of weakly symmetric
pseudo-Riemannian manifolds.  After the definition, we give a short proof 
of homogeneity in Proposition \ref{homog}.  In Proposition \ref{interchange}
we describe circumstances under which a pair of points can be
interchanged by an isometry.  Example \ref{no-join} exhibits a class of
symmetric (thus weakly symmetric) pseudo-Riemannian manifolds in which not
every pair of points is joined by an unbroken geodesic.
We then turn to weak symmetries and weakly symmetric coset spaces. 
In Proposition \ref{coset-riem} we recall their relation to weakly 
symmetric Riemannian manifolds, and in Proposition \ref{coset-p-riem}
we extend this to the setting of weakly symmetric pseudo-Riemannian
manifolds.
\medskip

Section \ref{examples} presents a number of examples of weakly symmetric
pseudo-Riemannian manifolds related in one way or another to weakly
symmetric Riemannian manifolds.
\medskip

In Section \ref{geodesics} we study geodesics and prove the
geodesic orbit property for weakly symmetric pseudo-Riemannian manifolds.
We note that the fact that the property that the nilradical of the
isometry group is 2-step nilpotent, which holds in the Riemannian case,
fails in the pseudo-Riemannian setting.  Then we give a condition under
which it holds for weakly symmetric pseudo-Riemannian manifolds.

\section{Pseudo-Riemannian manifolds}\label{prm}

We collect some basic facts on pseudo-Riemannian manifolds, 
following \cite{Ne1}.

\begin{definition}
A {\it pseudo-Riemannian manifold} $(M,\langle,\rangle)$ is a smooth manifold $M$ 
with a nondegenerate inner product $\langle, \rangle$ on the fibers of
its tangent bundle $TM$. Let the expression $(n_+, n_-)$, where 
$n_++n_-=\dim M$, denote the signature
of $\langle, \rangle$.  The manifold $(M,\langle,\rangle)$ Riemannian
in the case where $\langle,\rangle$ has signature $(\dim M,0)$, i.e. is
positive definite.
\hfill $\diamondsuit$\end{definition}

\begin{definition}
A {\it broken geodesic} is a piecewise smooth curve segment whose smooth 
subsegments are geodesics. A pseudo-Riemannian manifold is said to be 
{\it geodesically complete} if every maximal geodesic is defined on the entire real line.
\hfill $\diamondsuit$\end{definition}

\begin{lemma}\label{lem1}
A pseudo-Riemannian manifold is connected if and only if any two points can be connected by a broken geodesic.
\end{lemma}

In the connected Riemannian case we have a metric space structure where
the distance $d(x,y)$ is the infimum of the lengths of sectionally smooth
curves joining $x$ to $y$.  The classical Hopf-Rinow theorem is
\begin{theorem}
For a connected Riemannian manifold $M$, the following conditions are 
equivalent.

$(1)$ As a metric space under the Riemannian metric  $M$ is complete, i.e., 
every Cauchy sequence converges.

$(2)$ There exists a point $x\in M$ from which $M$ is geodesically complete,
i.e.,  $exp_x$ defines on the entire tangent space $T_xM$.

$(3)$ $M$ is geodesically complete.

$(4)$ Every closed bounded subset of $M$ is compact.
\end{theorem}
Furthermore,
\begin{proposition}\label{prop1}
If a connected Riemannian manifold is complete, then any two of its points 
are joined by a minimizing geodesic segment.
\end{proposition}

\begin{definition}
Let $(M,\langle,\rangle _M)$ and $(N,\langle,\rangle _N)$ be 
pseudo-Riemannian manifolds. An {\it isometry} from $M$ to $N$ is a diffeomorphism $\phi: M\rightarrow N$ satisfying $\langle d\phi(v),d\phi(w)\rangle _N=\langle v,w\rangle _M$ for any $v,w\in T_xM$ and $x\in M$.
\hfill $\diamondsuit$\end{definition}

\begin{definition}
A pseudo-Riemannian manifold $(M,\langle,\rangle)$ is {\it homogeneous} if for any points $x$ and $y$ there exists an isometry $\phi$ of $M$ such that $\phi(x)=y$.
\hfill $\diamondsuit$\end{definition}

Let $I(M,\langle,\rangle)$ denote the isometry group of the pseudo-Riemannian 
manifold $(M,\langle,\rangle)$. Then $I(M,\langle,\rangle)$, with the
compact-open topology, is a Lie group.   A pseudo-Riemannian manifold 
$(M,\langle,\rangle)$ is homogeneous if $I(M,\langle,\rangle)$ is transitive on $M$.

\begin{lemma}
A Riemannian homogeneous manifold is complete.  There exist 
homogeneous pseudo-Riemannian manifolds that are not geodesically complete.
\end{lemma}
Together with Proposition~\ref{prop1}, we have:
\begin{lemma}\label{lem3}
Any two points of a connected Riemannian homogeneous manifold is joined by 
a minimizing geodesic segment.  This fails for some pseudo-Riemannian manifolds.
\end{lemma}

\begin{lemma}\label{lem2}
If a Lie group acts transitively on a connected manifold, then so does 
the identity component of the Lie group.  In particular if
$(M,\langle,\rangle)$ is a connected homogeneous pseudo-Riemannian
manifold then the identity component $I(M,\langle,\rangle)^0$ of its
isometry group is transitive.
\end{lemma}

\section{Geometry of pseudo-Riemannian weakly symmetric manifolds}
\label{geom}

There are a number of equivalent conditions that can be taken as the
definition of weak symmetry for a Riemannian manifold.  The one on
reversing geodesics is also appropriate in the pseudo-Riemannian case.
\begin{definition}
Let $(M,\langle,\rangle)$ be a pseudo-Riemannian manifold. Suppose that 
for every $x\in M$ and every nonzero tangent vector $\xi\in T_xM$, there 
is an isometry $\phi=\phi_{x,\xi}$ of $(M,\langle,\rangle)$ such that 
$\phi(x)=x$ and $d\phi(\xi)=-\xi$. Then we say that $(M,\langle,\rangle)$ 
is a {\it pseudo-Riemannian weakly symmetric manifold}. In particular, 
if $(M,\langle,\rangle)$ is a Riemannian manifold, then we say that 
$(M,\langle,\rangle)$ is weakly symmetric.
\hfill $\diamondsuit$\end{definition}

Of course the symmetric case is the case where each $\phi_{x,\xi}$ is
independent of $\xi$, in other words where for any $x\in M$ the geodesic 
reflection at $x$ extends to a globally defined isometry of $M$. 
Equivalently, $(M,\langle,\rangle)$ is a pseudo-Riemannian symmetric space 
if for any $x\in M$ there is an involutive isometry $\theta_x$ of $M$
that has $x$ as an isolated fixed point.  

As in the Riemannian case we have

\begin{proposition}\label{homog}
Let $(M,\langle,\rangle)$ be a connected pseudo-Riemannian weakly symmetric 
manifold.  Then $(M,\langle,\rangle)$ is a pseudo-Riemannian homogeneous 
space $G/H$ where $G=I(M,\langle,\rangle)^0$.
\end{proposition}
\begin{proof}
For any $x,y\in M$, by Lemma~\ref{lem1}, there exists a broken geodesic 
connecting $x$ and $y$. Assume that the broken geodesic is 
$o'_1o'_2\cdots o'_p$, where $o'_io'_{i+1}$ is a geodesic segment for any 
$1\leq i\leq p-1$, $o'_1=x$ and $o'_p=y$. Let $\gamma_i$ denote the geodesic 
arc from $o'_i$ to $o'_{i+1}$ parameterized from $0$ to $1$ proportional 
to the arc length for any $i$. Since $(M,\langle,\rangle)$ is weakly 
symmetric, we have that there exists $g_i\in I(M,\langle,\rangle)$ such 
that $g_i\gamma_i(\frac{1}{2})=\gamma_i(\frac{1}{2})$ and 
$dg_i\gamma_i^\prime(\frac{1}{2})=-\gamma_i^\prime(\frac{1}{2})$.  
Then $g_i$ interchanges $o'_i$ and $o'_{i+1}$. Thus $g_{p-1}\cdots g_1(x)=y$. 
So $I(M,\langle,\rangle)$ is transitive on $M$. By Lemma~\ref{lem2}, 
$I(M,\langle,\rangle)^0$ acts transitively on $M$. So
$(M,\langle,\rangle)$ is homogeneous.
\end{proof}

Let $(M,\langle,\rangle)$ be a connected Riemannian weakly symmetric manifold. 
It is homogeneous, hence complete, so any two points in $M$ are 
connected by a geodesic segment. Thus for any $x,y\in M$ there exists an 
involutive isometry (the geodesic symmetry at the midpoint of that geodesic
segment) that interchanges $x$ and $y$.  In the pseudo-Riemannian case,
although we know that $(M,\langle,\rangle)$ is homogeneous, there might
not be a geodesic joining any two points.  So we only have 

\begin{proposition}\label{interchange}
Let $(M,\langle,\rangle)$ be a connected pseudo-Riemannian manifold. 
If $(M,\langle,\rangle)$ is weakly symmetric, then for any $x,y\in M$ 
connected by a geodesic, there exists an isometry which interchanges 
$x$ and $y$. In particular if $(M,\langle,\rangle)$ is Riemannian, then for any $x,y\in M$ there exists an isometry which interchanges $x$ and $y$.
\end{proposition}

\begin{exam}\label{no-join}
Here is a class of pseudo-Riemannian symmetric manifolds that have 
points which cannot be joined by a single geodesic.  Let $G$ be a semisimple 
Lie group such that the exponential map $exp: Lie(G)\rightarrow G$ is not 
surjective, and let $x$ be a point of $G$ that is not in the image of 
$exp: Lie(G)\rightarrow G$. Such Lie groups exist, for example if 
$G = SL(3;R)$, see \cite{DT1}. Use the Killing form for the 
pseudo-Riemannian metric. Then $G$ should be a pseudo-Riemannian symmetric 
space, where the symmetry at the identity element is $g \mapsto g^{-1}$ , 
and the geodesics are the group translates of the orbits of one-parameter 
subgroups. See Theorem \ref{thm3} below.  
This shows that there is no geodesic from the identity element to the point $x$.
\hfill $\diamondsuit$\end{exam}

\begin{definition}
Let $G$ be a connected Lie group and $H$ be a closed subgroup. Suppose 
that $\sigma$ is an automorphism of $G$ such that $\sigma(g)\in Hg^{-1}H$ 
for every $g\in G$. Then we say that $G/H$ is a {\it weakly symmetric coset 
space}, that $(G,H)$ is a {\it weakly symmetric pair}, and that 
$\sigma$ is a {\it weak symmetry} of $G/H$.
\hfill $\diamondsuit$\end{definition}

Consider the case where $H$ is a compact subgroup. The next proposition 
recalls the relation between group-theoretic notion and differential-geometric 
notions of weak symmetry. There the condition of $g(x)=\phi(y)$ and 
$g(y)=\phi(x)$ corresponds to the interchange condition between pairs of points.

\begin{proposition}[\cite{Wo1}]\label{coset-riem}
Let $G$ be a connected Lie group, $K$ be a compact subgroup and $M=G/K$.

$(1)$ $G/K$ is a weakly symmetric coset space if and only if there is a 
diffeomorphism $\phi$ of $M$ such that $\phi G\phi^{-1}=G$ and 
if $x,y\in M$ there exists $g\in G$ with $g(x)=\phi(y)$ and $g(y)=\phi(x)$.

$(2)$ Assume that the above equivalent conditions hold. Then every 
$G$-invariant Riemannian metric on $M$ is $\phi$-invariant and weakly 
symmetric. In particular there exists a $G$-invariant $\phi$-invariant 
Riemannian metric on $M$.
\end{proposition}

More generally, we have
\begin{proposition}\label{coset-p-riem}
Let $G$ be a connected Lie group, $H$ be a closed subgroup and $M=G/H$.

$(1)$ If there is a diffeomorphism $\phi$ of $M$ such that 
$\phi G\phi^{-1}=G$ and if $x,y\in M$ there exists $g\in G$ with 
$g(x)=\phi(y)$ and $g(y)=\phi(x)$, then $G/H$ is a weakly symmetric coset space.

$(2)$ If $G/H$ is a weakly symmetric coset space, then there is a 
diffeomorphism $\phi$ of $M$ such that $\phi G\phi^{-1}=G$ and if 
$x,y\in M$ connecting by a geodesic there exists $g\in G$ with 
$g(x)=\phi(y)$ and $g(y)=\phi(x)$. Moreover, every $G$-invariant 
pseudo-Riemannian metric on $M$ is $\phi$-invariant and weakly symmetric.

$(3)$ If there is a diffeomorphism $\phi$ of $M$ such that 
$\phi G\phi^{-1}=G$ and if $x,y\in M$ connecting by a geodesic there exists 
$g\in G$ with $g(x)=\phi(y)$ and $g(y)=\phi(x)$, then for every $g\in G$ 
there is $h\in G$ and an automorphism $\sigma$ of $G$ such that 
$\sigma(h^{-1}gh)\in Hh^{-1}g^{-1}hH$.
\end{proposition}
\begin{proof}
The first two assertions follow from the proof for the Riemannian case. 
For the third one, assume that there exists a diffeomorphism $\phi$ such 
that the conditions of $(3)$ hold. Translating by an element of $G$ we 
may assume that $\phi(1H)=1H$. Then $\sigma: g\mapsto \phi g\phi^{-1}$ 
defines an automorphism $\sigma$ of $G$ such that $\sigma(H)=H$ and 
$\phi(gH)=\sigma(g)H$. Fix $g\in G$ and $g\not\in H$. Let $g'H$ be a fixed 
point of $g$ in $G/H$ and $U$ be a normal neighborhood of $g'H$. Then 
$U\cap gU$ is a neighborhood of $g'H$. Then there exists $g_1H\in U$ such 
that $g_1H\not=gg_1H$ and $g(g_1H)\in U\cap gU$. That is, there is a 
geodesic connecting $g_1H$ and $g(g_1H)$. Then $1H$ and $g_1^{-1}gg_1H$ 
are connected by a geodesic. By the assumption, there exists $g_2\in G$ 
such that $$g_2g_1^{-1}gg_1H=s(1H)=1H \text{ and } g_2(1H)=s(g_1^{-1}gg_1H).$$ 
It follows that $\sigma(g_1^{-1}gg_1)\in Hg_1^{-1}g^{-1}g_1H$.
\end{proof}

\section{Examples of pseudo-Riemannian weakly symmetric spaces}\label{examples}

We describe a number of weakly symmetric pseudo-Riemannian manifolds.  The
reductive ones will be familiar to many readers.

\begin{exam}\label{exam1}$M_{p,q}=G_{p,q;{\mathbb C}}/U(p,q;{\mathbb C})$ 
where $G_{p,q;{\mathbb C}}$ is given as follows.  The space ${\mathbb C}^n$ 
of $n$-tuples over $\mathbb C$ is viewed as a right vector space (so that
it is easy to extend considerations from ${\mathbb C}$ to the quaternions
and the octonions). Scalars act on the right and linear transformations act 
on the left. For any non-negative integers $p$ and $q$ and $n=p+q$, we have 
the Hermitian vector space ${\mathbb C}^{p,q}$ with the Hermitian form 
$$
\langle v,w\rangle_1=\sum^p_{i=1}v^i\overline{w}^i-
	\sum^q_{i=1}v^{p+i}\overline{w}^{p+i}.
$$
Its unitary group is $U(p,q;{\mathbb C})$. We have a Heisenberg 
group $H_{p,q;{\mathbb C}}$ which is the real vector space 
${\rm Im} {\mathbb C}+{\mathbb C}^{p,q}$ with the group composition
$$(v,w)(v',w')=(v+v'+{\rm Im}h(w,w'),w+w').$$
Then $g(v,w)=(v,g(w))$ defines an action of the unitary group 
$U(p,q;{\mathbb C})$ by automorphisms on $H_{p,q;{\mathbb C}}$. 
The semidirect product group 
$G_{p,q;{\mathbb C}}:=H_{p,q;{\mathbb C}}\rtimes U(p,q;{\mathbb C})$ 
has group composition
$$(v,w,g)(v',w',g')=(v+v'+{\rm Im}h(w,g(w')),w+g(w'),gg').$$
It is clear that $H_{p,q;{\mathbb C}}$ is 2-step nilpotent. If $q=0$, 
denote $H_{n,0;{\mathbb C}}$ and $G_{n,0;{\mathbb C}}$ by 
$H_{n;{\mathbb C}}$ and $G_{n;{\mathbb C}}$ respectively. The usual 
Heisenberg group $H_{\mathbb C}$ is $H_{n;{\mathbb C}}$. For details 
and extension to quaternionic and octonionic Heisenberg groups 
see \cite{Wo1}. 
\medskip

The action of $G_{p,q;{\mathbb C}}$ on 
$M_{p,q}:=G_{p,q;{\mathbb C}}/U(p,q;{\mathbb C})$ is transitive. At 
$x:=(1,0,\cdots,0)$ the isotropy subgroup of $G_{p,q;{\mathbb C}}$ 
is $U(p,q;{\mathbb C})$. The tangent space at $x$ can be viewed as 
${\rm Im} {\mathbb C}+{\mathbb C}^{p,q}$, and $dg(v,w)=(v,gw)$ for any 
$(v,w)\in T_xM_{p,q}$ and $g\in U(p,q;{\mathbb C})$. Any 
$G_{p,q;{\mathbb C}}$-invariant pseudo-Riemannian metric is determined by a 
$U(p,q;{\mathbb C})$-invariant inner product on $T_xM_{p,q}$. The action of 
$U(p,q;{\mathbb C})$ on ${\mathbb C}^{p,q}$ is absolutely irreducible, and
it is and trivial on ${\rm Im}{\mathbb C}$, so every 
$U(p,q;{\mathbb C})$-invariant inner product $\langle, \rangle$  on
$T_xM_{p,q}$ satisfies $\langle {\rm Im}{\mathbb C}, {\mathbb C}^{p,q}\rangle
= 0$.  Let $\langle v,v'\rangle_1 = v\bar v'$,  the usual inner product on 
${\rm Im}{\mathbb C}$, and $\langle w, w' \rangle_2 = {\rm Re\,}h(w,w')$, the
usual (real) inner product on ${\mathbb C}^{p,q}$. Then $\langle, \rangle$ 
can be any real linear combination $a\langle , \rangle_1 \oplus
b\langle , \rangle_2$, $a \ne 0 \ne b$.  That gives all the 
$G_{p,q;{\mathbb C}}$-invariant pseudo-Riemannian metrics on $M_{p,q}$.
\medskip

The map $\phi$ given by conjugation on each coordinate is an isometry of 
$M_{p,q}$.  More precisely, 
$$
\phi(x)=x \text{ and } d\phi(v,w)=(-v,\sqrt{-1}w).
$$
Furthermore, there exists an element $g_w$ in $U(p,q;{\mathbb C})$ satisfying 
$dg_w(-v,\sqrt{-1}w)=(-v,-w)$ since $U(p,q;{\mathbb C})$ acts transitively on 
any cone or quadric 
$\{w \in {\mathbb C}^{p,q}\setminus \{0\} \mid h(w,w) = r\}$.
That is, the isometry $g_w\cdot\phi$ satisfies
$$g_w\cdot\phi(x)=x \text{ and } dg_w\cdot d\phi(v,w)=(-v,-w).$$
Since the isotropy group at different points are conjugate in 
$G_{p,q,{\mathbb C}}$ for a homogeneous space, we have:
\begin{enumerate}
\item $(M_{p,q},\langle,\rangle)$ is a pseudo-Riemannian weakly symmetric 
manifold.
\item $(M_{p,q},\langle,\rangle)$ is a Riemannian weakly symmetric manifold 
if and only if, either $q=0$ and $a,b$ are positive, or $p=0$ and $a$ is 
positive and $b$ are negative.  \hfill $\diamondsuit$
\end{enumerate}
\end{exam}
\medskip

The variation of pseudo-Riemannian weakly symmetric structure just
described, is a simple group-theoretic phenomenon.
Let $(M,\langle,\rangle)$ be a pseudo-Riemannian weakly symmetric manifold. 
Express $M = G/H$ where $G = I(M,\langle,\rangle)$.  Then, from the action
of the isotropy subgroup $H$ on the tangent space to $M$ at $1H$ we see
that every $G$--invariant pseudo-Riemannian metric on $M$ is weakly symmetric.
We will see additional instances of this in Examples \ref{exam2}, \ref{exam3}
and \ref{exam4}.

\begin{exam}\label{exam2}
Look at the transitive action of $U(n)$ on 
$M=S^{2n-1}(1)\subset {\mathbb C}^{n}$. If $x=(1,0,\cdots,0)$, then 
$H_x=U(n-1)$. Here $T_xM$ can be view as 
${\rm Im} {\mathbb C}+{\mathbb C}^{n-1}$ and $$dg(v,w)=(v,gw)$$ for any 
$g\in H_x$ and $(v,w)\in {\rm Im} {\mathbb C}+{\mathbb C}^{n-1}$. Any 
$U(n)$-invariant pseudo-Riemannian metric on $M$ is determined by a 
$U(n-1)$-invariant inner product on the tangent space.  Let 
$\langle,\rangle_1$ and $\langle,\rangle_2$ be the usual inner products on 
${\rm Im}{\mathbb C}$ and ${\mathbb C}^{n-1}$ respectively. Then 
$$\langle (v,w),(v',w')\rangle=a\langle v,v'\rangle_1+b\langle w,w'\rangle_2,
\,\, a \text{ and } b \text{ nonzero constants},$$ 
is a $U(n-1)$-invariant inner product that induces a $U(n)$-invariant 
pseudo-Riemannian metric on $M$. The map $\phi$ defined by conjugation on 
each coordinate is an isometry of $M$. Then $\phi(x)=x$ and 
$d\phi(v,w)=(-v,\sqrt{-1}w)$ for any $(v,w)\in T_xM$. Since $U(n-1)$ acts 
transitively on the unit sphere in ${\mathbb C}^{n-1}$, it contains an 
element $g_w$ that sends $\sqrt{-1}w$ to $-w$. Hence if $G=U(n)\cup \phi U(n)$, 
then $H_x=U(n-1)\cup \phi U(n-1)$, and if $(v,w)\in T_xM$ then there 
exists $k_w = g_w\cdot\phi\in U(n-1)$ such that $$k_w(v,w)=(-v,-w).$$ 
Hence the $G$-invariant 
pseudo-Riemannian metric $\langle,\rangle$ on $M$ is weakly symmetric,
and of course is Riemannian weakly symmetric if $a, b > 0$. 
\hfill $\diamondsuit$\end{exam}

\begin{exam}\label{exam3}
Consider the transitive action of $Sp(1)\times Sp(n)$ on 
$M=S^{4n-1}(1)\subset {\mathbb H}^n$ given by $(g_1,g_2)(x)=g_2(x)g_1^{-1}$. 
The isotropic group of $x=(1,0,\cdots,0)$ is equal to 
$\Delta Sp(1)\times Sp(n-1)$, here $\Delta Sp(1)$ means that $Sp(1)$ 
embeds diagonally. The tangent space $T_xM$ can be identified with 
${\rm Im} {\mathbb H}+{\mathbb H}^{n-1}$, and 
$$
d(g_1,g_2)(v,w)=(g_1vg_1^{-1},g_2(w))
$$ 
for any $(g_1,g_2)\in \Delta Sp(1)\times Sp(n-1)$. Since 
$v\mapsto g_1vg_1^{-1}$ is the standard action of $Sp(1)/{\mathbb Z}_2=SO(3)$ 
on ${\mathbb R}^3$ and $Sp(n-1)$ acts transitively on the unit sphere in 
${\mathbb H}^{n-1}$, we have that for any $(v,w)\in T_x(M)$ there exists 
$(g_1,g_2)\in \Delta Sp(1)\times Sp(n-1)$ such that $d(g_1,g_2)(v,w)=(-v,-w).$ 
Hence any $(Sp(1)\times Sp(n))$-invariant pseudo-Riemannian metric on $M$ 
is weakly symmetric. Let $\langle,\rangle_1$ and $\langle,\rangle_2$ be 
$Sp(1)$- and $Sp(n-1)$-invariant metrics on ${\rm Im}{\mathbb H}$ and 
${\mathbb H}^{n-1}$ respectively. As before, if $a \ne 0 \ne b$ then
$a\langle,\rangle_1+b\langle,\rangle_2$ defines a weakly symmetric
$(Sp(1)\times Sp(n))$-invariant pseudo-Riemannian metric on $M$, which of 
course is Riemannian if and only if $a, b > 0$.
\hfill $\diamondsuit$\end{exam}

\begin{exam}\label{exam4}
Let $M=S^{15}=Spin(9)/Spin(7)$ with the tangent space 
$T_x(M)={\rm Im} {\mathbb O}+{\mathbb O}$. For any $g\in Spin(7)$, 
$dg(v,w)=(\rho(g)(v),\varphi(g)(w))$, where $\rho$ is the standard 
representation of $Spin(7)$ on ${\rm Im} {\mathbb O}$ (via the two fold 
cover $Spin(7) \to SO(7)$), and where $\varphi$ is the spin representation 
of $Spin(7)$ on ${\mathbb O}$.  If 
$(v,w)\in {\rm Im} {\mathbb O}+{\mathbb O}$, we have 
$g_1\in Spin(7)$ such that $\rho(g_1)v=-v$. The isotropy subgroup of $Spin(7)$ 
at $v$ (under the action of $\rho$) is $Spin(6)=SU(4)$, and the restriction 
of the action of $Spin(7)$ on ${\mathbb O}$ to $SU(4)$ is the standard action 
of $SU(4)$ on ${\mathbb O}$.  That is transitive on the unit sphere, so 
there exists an element $g_2\in Spin(7)$ such that 
$$
\rho(g_2)v=v \text{ and }\varphi(g_2)\varphi(g_1)w=-w.
$$ 
In other words, $d(g_1,g_2)(v,w)=(-v,-w)$. Let $\langle,\rangle_1$ and 
$\langle,\rangle_2$ be the $Spin(7)$-invariant inner products on 
${\rm Im}{\mathbb O}$ and ${\mathbb O}^{n-1}$ respectively. 
Then $a\langle,\rangle_1+b\langle,\rangle_2$ is a $Spin(7)$-invariant inner
product on the tangent space of $M=S^{15}$. As before, if
$a \ne 0 \ne b$ it induces a weakly symmetric $Spin(9)$-invariant 
pseudo-Riemannian metric on $M$. If $a, b > 0$ that metric is
Riemannian.
\hfill $\diamondsuit$\end{exam}

\section{Geodesics in pseudo-Riemannian weakly symmetric spaces}
\label{geodesics}
In this section we discuss questions of completeness and the geodesic 
orbit property, and implications for the structure of the nilradical
of the isometry group.

\begin{proposition}\label{prop3}
Any pseudo-Riemannian weakly symmetric space is geodesically complete.
\end{proposition}
\begin{proof}
It is enough to show that a geodesic $\gamma: [0,a)\rightarrow M$ 
is extendible. Choose $b$ near $a$ in the interval, let $g$ be the isometry 
satisfying $g(\gamma(b))=\gamma(b)$ and $g(\gamma'(b))=-\gamma'(b)$. Since 
$g$ reverses the geodesic through $\gamma(b)$, a reparameterization of 
$g\circ \gamma$ provides the required extension of $\gamma$.
\end{proof}

The main theorem in \cite{BKV} is that any maximal geodesic in a Riemannian 
weakly symmetric space $M$ is an orbit of a one-parameter group of 
isometries of $M$.  Proposition~\ref{prop3} lets us follow the
argument of \cite{BKV} and push it to the pseudo-Riemannian setting.  Thus

\begin{theorem}\label{thm3}
Any maximal geodesic in a pseudo-Riemannian weakly symmetric space $M$ is 
an orbit of a one-parameter group of isometries of $M$.
\end{theorem}

\begin{definition}
A connected pseudo-Riemannian homogeneous manifold $M$ is said to be a 
{\it geodesic orbit space} if every maximal geodesic in $M$ is an orbit of a 
one-parameter group of isometries of $M$.
\hfill $\diamondsuit$\end{definition}

It is immediate from the definition that a connected pseudo-Riemannian 
geodesic orbit space is homogeneous. 
\medskip

A connected pseudo-Riemannian manifold which admits a transitive nilpotent 
group of isometries is said to be a {\it nilmanifold}.  In the Riemannian case, 
we have:
\begin{theorem}[\cite{G1}]\label{thm1}
Let $M$ be a geodesic orbit Riemannian nilmanifold, say with transitive
nilpotent group $N$ of isometries.  Then $N$ is at most $2$-step nilpotent.
\end{theorem}

By a small extension of Theorem~\ref{thm1}, for geodesic orbit Riemannian 
nilmanifolds, we have:

\begin{theorem}[\cite{Wo1}]\label{thm2}
Let $(M,\langle,\rangle)$ be a connected and simply connected Riemannian 
geodesic orbit space, $G=I(M,\langle,\rangle)^0$, and $N$ the nilradical of 
$G$. Then $N$ is at most $2$-step nilpotent.
\end{theorem}

Combining Theorems~\ref{thm1} and \ref{thm2}, we have:

\begin{theorem}\label{thm2a}
Let $(M,\langle,\rangle)$ be a connected and simply connected Riemannian 
weakly symmetric manifold, $G=I(M,\langle,\rangle)^0$, and $N$ the nilradical 
of $G$. Then $N$ is at most $2$-step nilpotent.
\end{theorem}
This result is due independently to Benson-Ratcliff, Gordon and Vinberg.
See \cite{Wo1} for details.

\begin{exam}\label{exam5}
This is an example of Kath and Olbrich \cite[Example 4.6]{KO} which shows 
that Theorem \ref{thm2a} fails dramatically in the pseudo-Riemannian case.
For each positive integer $m$ define a real Lie algebra
$
\mathfrak{g} = {\mathbb R}^m \oplus {\mathbb C}^{m+1}
$.  
We have the standard real basis $\{z_k\}_{1\leqq k \leqq m}$ of ${\mathbb R}^m$
and the standard complex basis $\{e_i\}_{1\leqq i \leqq m+1}$ of
${\mathbb C}^{m+1}$, and thus a real basis $\{z_k\}\cup \{e_i\} \cup \{f_j\}$
of $\mathfrak{g}$ where $f_j = \sqrt{-1}\, e_j$.  The nonzero brackets
between the basis vectors are
$$
[e_i,f_j] = z_{i+j-1}, \,\, [z_k,e_i] = f_{i+k}, \,\, [z_k,f_j] = -e_{k+j}
$$
where $z_\ell = e_q = f_q = 0$ for $\ell > m, q > m+1$.  
As usual let $\mathfrak{g} = \mathfrak{g}^0$ and
$\mathfrak{g}^{s+1} = [\mathfrak{g}, \mathfrak{g}^s]$.
Then $\mathfrak{g}^{2r-1}$ is spanned by $\{z_k, e_i, f_j\}$ with
$r+1 \leqq i,j \leqq m+1$ and $r \leqq k \leqq m$ and
$\mathfrak{g}^{2r}$ is spanned by $\{z_k, e_i, f_j\}$ with
$r+1 \leqq i,j \leqq m+1$ and $r+1 \leqq k \leqq m$.  Thus $\mathfrak{g}$
is $(2m+1)$-step nilpotent.
\medskip

The inner product
on $\mathfrak{g}$ is given by ${\mathbb R}^m \perp {\mathbb C}^{m+1}$ and
$$
\langle e_i, f_j \rangle = 0, \,\, \langle e_i, e_j\rangle =
\delta_{i+j,m+2} = \langle f_i,f_j\rangle, \,\, \langle z_k, z_\ell \rangle
= \delta_{k+\ell, m+1}\,\, .
$$
Consider the corresponding pseudo-Riemannian metric on the connected 
simply connected nilpotent Lie group $G$ with Lie algebra $\mathfrak{g}$.
Compute $\langle x, [y,z]\rangle = \langle [x,y],z\rangle$ to see that the
corresponding pseudo-Riemannian metric is bi-invariant.  Thus $G$ is a
symmetric pseudo-Riemannian nilmanifold which, as  Lie group, 
is $(2m+1)$-step nilpotent.
\hfill $\diamondsuit$\end{exam}

We now give a sufficient condition for $N$ to be 
$2$-step nilpotent in the pseudo-Riemannian cases.
\medskip

Let $(M,\langle,\rangle$ be a pseudo-Riemannian homogeneous manifold. 
Let $G\subset I(M)^0$ be a closed subgroup that acts transitively on $M$. 
Let $p\in M$ and $H$ the isotropy group of $G$ at $p$, so $M$ 
can be identified with $G/H$. The pseudo-Riemannian metric 
$\langle, \rangle$ on $M$ can be viewed as a $G$-invariant
metric on $G/H$. 
\medskip

If the metric $\langle,\rangle$ is positive definite, then 
$(G/H, \langle,\rangle)$ is a reductive homogeneous space. If the 
metric $\langle,\rangle$ is indefinite, a reductive decomposition need 
not exist. See \cite{FMP1} for an example of nonreductive pseudo-Riemannian 
homogeneous space.  Fix a fixed reductive decomposition 
${\mathfrak g}={\mathfrak m}+{\mathfrak h}$.  We
identify ${\mathfrak m}\subset {\mathfrak g}=T_eG$ with the 
tangent space $T_pM$ via the projection
$\pi: G\rightarrow G/H=M$. Using this identification we view the scalar 
product $\langle,\rangle_p$ on $T_pM$ as an $Ad(H)$-invariant scalar 
product on ${\mathfrak m}$.
\medskip

The definition of a homogeneous geodesic is well-known in the Riemannian case 
(see, e.g., \cite{KV1}). In the pseudo-Riemannian case, the necessary 
generalized version was given in \cite{DK1}:

\begin{definition}
Let $M = G/H$ be a pseudo-Riemannian reductive homogeneous space, 
${\mathfrak g}={\mathfrak m}+{\mathfrak h}$ a reductive decomposition, 
and $p$ the base point of $G/H$.  Let $s \mapsto \gamma(s)$ be a geodesic 
through $p$ with affine parameter $s$ in an open interval $J$.  Then
$\gamma$ is {\it homogeneous} if there exist

1) a diffeomorphism $s=\phi(t)$ between the real line and the open interval $J$
and

2) a vector $X\in {\mathfrak g}$ such that $\gamma(\phi(t))=exp(tX)(p)$ 
for all $t\in {\mathbb R}$.

\noindent
The vector $X$ is then called a {\it geodesic vector}.
\hfill $\diamondsuit$\end{definition}

\begin{remark}
In the Riemannian situation or the pseudo-Riemannian weakly symmetric space 
cases, the diffeomorphism from the condition 1 is always the identity map 
on the real line and hence the definition can be formulated more simply.
\hfill $\diamondsuit$\end{remark}

The basic formula characterizing geodesic vectors in the pseudo-Riemannian
case appeared in \cite{FMP1} and \cite{P1}, but without a proof. The correct 
mathematical formulation with the proof was given in \cite{DK1}:
\begin{lemma}(Geodesic Lemma). Let $M=G/H$ be a reductive pseudo-Riemannian homogeneous space, ${\mathfrak g}={\mathfrak m}+{\mathfrak h}$ a reductive decomposition and $p$ the base
point of $G/H$. Let $X\in {\mathfrak g}$. Then the curve $\gamma(t)=exp(tX)(p)$ (the orbit of
a one-parameter group of isometries) is a geodesic curve with respect to some
parameter $s$ if and only if
$$\langle [X,Z]_{\mathfrak m}, X_{\mathfrak m}\rangle=k\langle X_{\mathfrak m}, Z\rangle$$
for all $Z\in {\mathfrak m}$, where $k\in {\mathbb R}$ is a constant. Further, if $k=0$, then $t$ is an affine parameter for this geodesic. If $k\not=0$, then
$s=e^{-kt}$ is an affine parameter for the geodesic. The second case can occur only if the curve $\gamma(t)$ is a null curve in a (properly) pseudo-Riemannian space.
\end{lemma}

The following proposition follows easily from the Geodesic Lemma.

\begin{proposition}\label{prop4}
Every geodesic in a reductive pseudo-Riemannian homogeneous space $M=G/H$ is an orbit of a one-parameter group of isometries of $G$ if and only if for each $X\in {\mathfrak m}$, there exists $A\in {\mathfrak h}$ such that $\langle [X+A,Z]_{\mathfrak m}, X\rangle=k\langle X,Z\rangle$ for any $Z\in {\mathfrak m}$. In particular, $M$ is a geodesic orbit manifold if and only if this condition holds for $G=I(M)^0$.
\end{proposition}

Now we adapt the argument of \cite{G1} to the pseudo-Riemannian case.

\begin{theorem}
Let $(M,\langle,\rangle)$ be a connected pseudo-Riemannian weakly symmetric space, $G=I(M,\langle,\rangle)^0$, and $N$ the nilradical of $G$. If there is a reductive decomposition ${\mathfrak g}={\mathfrak m}+{\mathfrak h}$ such that ${\mathfrak n}\subset {\mathfrak m}$ and the metric restriction on $[{\mathfrak n},{\mathfrak n}]$ is positive or negative definite, then $N$ is at most $2$-step nilpotent.
\end{theorem}
\begin{proof}
Let ${\mathfrak a}$ be the orthocomplement of $[{\mathfrak n},{\mathfrak n}]$ 
in ${\mathfrak m}$. By Theorem~\ref{thm3}, $M$ is a geodesic orbit space. 
Then by Proposition \ref{prop4}, for each $X\in {\mathfrak m}$, there exists 
$A\in {\mathfrak h}$ such that 
$\langle [X+A,Z]_{\mathfrak m}, X\rangle=k\langle X,Z\rangle$ for any 
$Z\in {\mathfrak m}$. Fix $\zeta\in {\mathfrak a}$. 
For any $\xi\in [{\mathfrak n},{\mathfrak n}]$, there exists $A_{\xi}$ such 
that
$$
\langle [\xi+A_{\xi},\zeta]_{\mathfrak m},\xi\rangle=k\langle\xi,\zeta\rangle=0.
$$
Since ${\mathfrak n}$ is an ideal of ${\mathfrak g}$ and ${\mathfrak m}$ is 
$Ad_G(H)$-invariant, we have $[{\mathfrak n},{\mathfrak n}]$ and then 
${\mathfrak a}$ are $Ad_G(H)$-stable. It follows that 
$\langle [\zeta,\xi]_{\mathfrak m}, \xi\rangle=0$. That is, 
$ad(\zeta)|_{[{\mathfrak n},{\mathfrak n}]}$ is skew-symmetric for any 
$\zeta\in{\mathfrak a}$ by the assumption. But $ad(\zeta)$ is nilpotent 
for any $\zeta \in {\mathfrak n}$. So $ad(\zeta)=0$ for 
$\zeta\in {\mathfrak n}\cap{\mathfrak a}$. Since 
${\mathfrak n}\cap{\mathfrak a}$ generates ${\mathfrak n}$ must have 
$[{\mathfrak n},[{\mathfrak n},{\mathfrak n}]]=0$. Thus 
${\mathfrak n}$ is at most $2$-step nilpotent.
\end{proof}

\section{Acknowledgments}
This work is supported by National Natural Science Foundation of China (No.11001133) and the Fundamental Research Funds for the Central Universities.

\end{document}